\documentclass{daj}

\usepackage{amsmath,amsfonts}
\usepackage{amsthm}
\usepackage{amssymb}
\usepackage{enumitem}
\usepackage[english]{babel} 

\numberwithin{equation}{section}
\numberwithin{figure}{section}
\theoremstyle{plain}
\newtheorem{thm}{\protect\theoremname}
\theoremstyle{plain}
\newtheorem{defn}[thm]{\protect\definitionname}
\theoremstyle{plain}  
\newtheorem{prop}[thm]{\protect\propositionname}
\theoremstyle{plain}
\newtheorem{lem}[thm]{\protect\lemmaname}

\usepackage{babel}
\providecommand{\propositionname}{Proposition}
\providecommand{\lemmaname}{Lemma}
\providecommand{\theoremname}{Theorem}
\providecommand{\definitionname}{Definition}

\dajAUTHORdetails{%
  title = {Decomposition of Random Walk Measures on the One-Dimensional Torus}, 
  author = {Tom Gilat},
  plaintextauthor = {Tom Gilat},
}   

\dajEDITORdetails{%
   year={2020},
   number={3},
   received={17 September 2019},   
   published={4 March 2020},  
   doi={10.19086/da.11888},       
}   

\begin{document}

\begin{frontmatter}[classification=text]


\author[tomg]{Tom Gilat}

\begin{abstract}
The main result of this paper is a decomposition theorem for a measure on the one-dimensional torus. Given a "sufficiently large" subset $S$ of the positive integers, an arbitrary measure on the torus is decomposed as the sum of two measures. The first one $\mu_1$ has the property that the random walk with initial distribution $\mu_1$ evolved by the action of $S$ equidistributes very fast. The second measure $\mu_2$ in the decomposition is concentrated on very small neighborhoods of a small number of points. 
\end{abstract}
\end{frontmatter}




\begin{section}{Introduction}
This paper is concerned with the dynamics of subsemigroups of the positive integers acting on the one-dimensional torus $\mathbb{R}/\mathbb{Z}$. This extensive line of research goes back to Furstenberg, who described the minimal sets of the action of the semigroup generated by two multiplicatively independent integers. They are finite periodic orbits and the whole torus.

Furstenberg also made several conjectures about such actions, which had an enormous impact on the field. Perhaps the most prominent of these asks for a classification of invariant measures on the torus under the action of the semigroup generated by 2 and 3 (or any other pair of multiplicatively independent integers). There has been some remarkable progress on this problem, but the conjecture is still wide open.

These problems become more manageable if one considers the action of "larger" semigroups. For example, Einsiedler and Fish gave a classification of invariant measures under the action of a semigroup with positive logarithmic density.

The main result of this paper is a decomposition theorem for a measure on the torus. Given a "sufficiently large" subset $S$ of the positive integers, an arbitrary measure on the torus is decomposed as the sum of two measures. The first one $\mu_1$ has the property that the random walk with initial distribution $\mu_1$ evolved by the action of $S$ equidistributes very fast. The second measure $\mu_2$ in the decomposition is concentrated on very small neighborhoods of a small number of points. 

The proof of the main result uses tools from additive combinatorics and builds on the work of Bourgain, Furman, Lindenstrauss and Mozes on the classification of stationary measures under the action of non-commuting toral automorphisms.
 
We define the general setting as follows: for $L>0$, let $S\subset [L,2L]$ be a set of natural numbers, $|S|>L^{\beta}$ for $0<\beta<1$, with $S$ being $(\widetilde{C},\lambda)$-regular (definition follows). The variables $L,\beta$ and $\widetilde{C},\lambda>0$ should be considered as global parameters and are referred to in the different theorems, propositions and lemmas, typically by giving thresholds on their values in the conditions of the statements. All the measures in this paper are Borel measures, with the topology of the measure space clear from the context. For countable spaces such as $\mathbb{N}$ the topology is the discrete one. In a minor abuse of terminology we refer to the Haar measure on $\mathbb{T}$ as the Lebesgue measure on $\mathbb{T}$.

For a non-empty set $S\subset \mathbb{N}$ we let
\begin{equation}
\nu_{S}=\frac{1}{|S|}\sum_{s\in S}\delta_{s}
\end{equation}
be the measure that averages over $S$. The set $S$ acts on the torus in the following standard way: $s.x=sx\ (\mathrm{mod}\ 1)$ for $s\in S$. For $s\in S$, let $T_s:\mathbb{T}\rightarrow\mathbb{T}$ be the mapping:
$T_s (x): x \mapsto s.x$ .

We denote by $\mathcal{P}(\cdot)$ the space of Borel probability measures on a topological space. For $\mu\in\mathcal{P}(\mathbb{T}), \nu\in\mathcal{P}(S)$, define the measure $\nu*\mu\in\mathcal{P}(\mathbb{T})$ as follows:
\begin{equation}
\nu*\mu=\sum_{s\in S}\nu(s)T_{s*}\mu,
\end{equation}
where $T_{s*}\mu(E)=\mu(T_s^{-1}(E))$ for a Borel set $E\subset\mathbb{T}$.

The following definition says what it means for a set $S$ to be $(\widetilde{C},\lambda$)-regular.
\begin{defn}
	Let $\widetilde{C}, \lambda>0$. We say that a set $S\subset[L,2L]\subset\mathbb{N}$ is $(\widetilde{C},\lambda)$-\emph{regular at scale} $r$, where $r$ is a positive real number, if  
	\begin{equation}
	|I\cap S|\leq \widetilde{C}\cdot\left(\frac{|I|}{L}\right)^{\lambda}|S|
	\end{equation}
	for any interval $I\subset[L,2L]\subset\mathbb{R}$ with $|I|\geq r$. By $|\cdot|$, we denote cardinality or the Lebesgue measure according to context.
\end{defn}	
\noindent If we say that a set is $(\widetilde{C},\lambda)$-regular, we mean that it is $(\widetilde{C},\lambda)$-regular at scale $1$.

We are ready to state the main theorem of this work, which is a decomposition theorem for a measure on the torus. In the statement, $\mu_1,\mu_2$ are non-negative Borel measures on $\mathbb{T}$, and $\mathsf{B}_{x,r}$ denotes an open ball in $\mathbb{T}$ with center at $x$ and radius $r$. A finite set $X\subset \mathbb{T}$ is $\delta$-\emph{separated} if for any $x,y\in X, d(x,y)>\delta$, where $d(\cdot,\cdot)$ is the usual metric on $\mathbb{T}$.
\begin{thm}\label{MainThm1}
	Let $\mu$ be a probability measure on $\mathbb{T}$. For every $\lambda,\beta>0$ there exist $L_1,\kappa,C,U,\tau_0>0,k\in\mathbb{N}$, such that if $L>L_1$, $S \subset [L,2L] $ is a $(\widetilde{C},\lambda)$-regular set for some $\widetilde C <L^{\tau_0}$, and $0<\tau<\tau_0$ with $|S|>L^\beta$, then there is a decomposition $\mu=\mu_{1}+\mu_{2}$ such that
	\begin{equation}\label{mu1conclusion}
	\left| \widehat{\nu_{S}^{*k}*\mu_{1}}(n)\right|\leq L^{-\tau} \qquad \forall 0\neq |n|<L^\tau,
	\end{equation}
	and there are finite subsets $X_{1},X_{2},...,X_{l}$ of $\mathbb{T}$, where $l<L^{C\tau}$, such that each $X_i$ is $\frac{1}{M}$-separated, $\mu_{2}$ is supported on $\bigcup_{i=1}^{l}\bigcup_{x\in X_{i}}\mathsf{B}_{x,\frac{1}{N}}$,
	where $N=L^{U}$ and $M=N^{1-\kappa}$, and $\mu_1$ is supported on the complement of the support of $\mu_2$.
\end{thm}
 
The usefulness of this theorem is due to the fact that $M<N$ in a well controlled manner. This implies a granulation phenomenon on the support of $\mu_2$, meaning that  $\mu_2(\mathbb{T})$ is larger, in a controlled manner, than the Lebesgue measure of the support of $\mu_2$. In addition, \ref{mu1conclusion} effectively describes $\mu_1$ as being "close to uniform" over its support.

In a follow-up paper we intend to show how Theorem \ref{MainThm1} can be used to prove effective equidistribution results in this context.

\end{section}
	
\begin{section}{Preliminaries from additive combinatorics}
The following inequality is due to Ruzsa. It is surprisingly useful, given how simple it is to prove.
\begin{lem}[{{\cite[Lemma 2.6]{tao_vu_2006} -- Ruzsa triangle inequality}}]\label{RuszaTriangleIneq}
	Let $A,B,C\subset G$ with $G$ any additive group and $C\neq\emptyset$. Then
	\[
	|A-B|\leq\frac{|A-C||B-C|}{|C|}
	\]
\end{lem}

We will need the following graph-theoretic result, closely connected to the Balog-Szemer\'edi-Gowers Theorem, due to Sudakov B., Szemer\'edi E. and Vu  V. H. .
\begin{thm}[\cite{MR2155059}, Lemma 4.2]\label{BSG_Thm}
	Let n and K be positive numbers, and let $\mathcal{G}=\mathcal{G}(A,B,E)$ be a (finite) bipartite graph, where $|B|\leq |A|=n$ and $|E|=n^{2}/K$. Then one can find $A'\subset A$ and $B'\subset B$ such that
	\begin{enumerate}
		\item $|A'|\geq n/(16K^{2})$ and $|B'|\geq n/(4K)$,
		\item $|(A'\times B')\cap E|\geq |A'||B'|/(4K)$
		\item for each $a\in A'$ and $b\in B'$, there are $n^{2}/(2^{12}K^{5})$ paths of length 3 whose two endpoints are a and b.
	\end{enumerate} 
\end{thm}	

We need two pieces of notation for the statement of the next lemma. They will be used throughout the paper. 
\begin{defn}
	Given a set $A\subset\mathbb{R}$ and $M\in \mathbb{R}$, let $\mathcal{N}(A;M)$ be the minimal number of open balls with radius $M>0$ (the center of a ball is any point in $\mathbb{R}$) that cover the set $A$.
\end{defn}
\begin{defn}
	Given $\mu\in\mathcal{P}(\mathbb{T})$ and a real number $\delta\geq 0$, let 
	\begin{equation}
	\mathcal{F}(\mu,\delta)=\left\{a\in\mathbb{Z}\backslash\{0\}:|\widehat{\mu}(a)|>\delta\right\}
	\end{equation}
\end{defn}

The following lemma allows us to extract from an initial set of high Fourier coefficients a set of relatively large Fourier coefficients that is stable with respect to subtraction. The sets are in a window around 0 and are viewed at some resolution $M$. The proof uses a counting argument which involves the graph-theoretic Lemma~\ref{BSG_Thm} as well as the algebraic nature of the ring of integers.
\begin{lem}\label{FourierBSG}
	For positive integer numbers $N,M$ and a probability measure $\mu$ on $\mathbb{T}$; let $A_{0}$ be an $M$-separated subset of $\mathcal{F}\left(\mu,\delta\right)\cap [-N,N]$.
	Let $R$ be a positive real number, and assume that 
	\[
	\mathcal{N}(\mathcal{F}\left(\mu,\delta^{2}/8\right)\cap [-2N,2N];M)\leq R|A_{0}|  .
	\]
	Then there exists a subset $A_{1}\subset A_{0}$ such that
	\begin{enumerate}
		\item $\left|\frac{1}{|A_{1}|}\sum_{a\in A_{1}}\widehat{\mu}(a)\right|\geq \frac{\delta}{2}$
		\item $|A_{1}|\gg |A_{0}|\delta^{2}$
		\item $\mathcal{N}(A_{1}-A_{1};M)\ll |A_{0}|R^{6}\delta^{-8}$ 
	\end{enumerate}	
\end{lem}
\begin{proof}
	By passing to a subset $A\subset A_{0}$ of size $|A|\geq \frac{|A_{0}|}{4}$ we may assume that $\mathrm{Re}(e^{i\theta}\cdot\widehat{\mu}(a))>\frac{\delta}{2}$, for some fixed $\theta\in [0,2\pi)$ and for every $a\in A$. Therefore
	\begin{equation}
	\left|\frac{1}{|A|}\sum_{a\in A}\widehat{\mu}(a)\right|>\frac{\delta}{2}.
	\end{equation}
	Let
	\begin{equation}
	\phi(x)=\sum_{a\in A} e_{a}(x),
	\end{equation}
	where $e_{a}(x)$ stands for $e^{-2\pi ax i}$. Note that
	\begin{equation}
	|\phi(x)|^{2}=(\sum_{a\in A}e_{a}(x))\cdot\overline{(\sum_{b\in A}e_{b}(x))}=\sum_{a,b\in A}e_{a-b}(x).
	\end{equation}
	We have that
	\begin{equation}
	\begin{split}
	\sum_{a,b\in A}\widehat{\mu}(a-b) & = \int_{\mathbb{T}}\sum_{a,b\in A}e_{a-b}(x)d\mu(x) \\
	& = \int_{\mathbb{T}}|\phi(x)|^{2}d\mu(x)  \\
	& \geq \left| \int_{\mathbb{T}}\phi(x)d\mu(x)\right|^{2} \\
	& =\left| \sum_{a\in A}\widehat{\mu}(a) \right|^{2}
	\end{split}
	\end{equation}
	where the inequality is due to the Cauchy-Schwarz inequality.
	Therefore
	\begin{equation}
	\label{}
	\frac{1}{|A|^{2}}\sum_{a,b\in A}\widehat{\mu}(a-b)\geq \left|\frac{1}{|A|}\sum_{a\in A}\widehat{\mu}(a)\right|^{2}\geq \frac{\delta^{2}}{4}
	\end{equation}
	Let
	\begin{equation}
	\begin{split}
	E & = \left\{(a,b)\in A\times A: \widehat{\mu}(a-b)>\frac{\delta^{2}}{8}\right\} \\
	& = \left\{(a,b)\in A\times A: a-b\in \mathcal{F}(\mu,\frac{\delta^{2}}{8})\cap[-2N,2N] \right\}
	\end{split}
	\end{equation}
	By Lemma \ref{MarkovLemma} we have that
	\begin{equation}\label{InitialEdgesDensity}
	|E| \geq \frac{\delta^{2}}{8}|A|^{2}
	\end{equation}	
	Define
	\begin{equation}
	\bar{A}=\left\{\left\lfloor\frac{a}{M}\right\rfloor\cdot M : a\in A\right\}+\{0,M\}
	\end{equation}
	Note that (recalling that $A$ is $M$-separated)
	\begin{equation}\label{TruncatedSetSizeIneq}
	|A|\leq |\bar{A}| \leq 2|A|
	\end{equation}
	Define
	\begin{equation}
	H = \left\{\left\lfloor\frac{h}{M}\right\rfloor\cdot M:h\in\mathcal{F}(\mu,\frac{\delta^{2}}{8})\cap[-2N,2N]\right\}+\{0,M\}
	\end{equation}
	Note that 
	\begin{equation}\label{Hbound}
	|H|\leq 4\mathcal{N}(\mathcal{F}\left(\mu,\delta^{2}/8\right)\cap [-2N,2N];M)\leq 4R|A_{0}|\leq 16R|\bar{A}|
	\end{equation}
	Next, we define
	\begin{equation}
	\bar{E}= \left\{(a,b)\in \bar{A}\times \bar{A}: a-b\in H\right\}
	\end{equation}
	Note that
	\begin{equation}\label{TruncatedSetEdgesIneq}
	|\bar{E}|\geq|E|
	\end{equation}
	By \ref{InitialEdgesDensity}, \ref{TruncatedSetSizeIneq}, \ref{TruncatedSetEdgesIneq} we have that
	\begin{equation}
	|\bar{E}|\geq\frac{\delta^{2}}{32}|\bar{A}|^{2}
	\end{equation}
	We choose subsets $A',B'\subset \bar{A}$ using Theorem \ref{BSG_Thm}, such that
	\begin{equation}
	|A'|\geq \frac{\delta^{2}}{2^{7}}|\bar{A}|\quad,\quad |B'|\geq\frac{\delta^{4}}{2^{14}}|\bar{A}|
	\end{equation} Let $a\in A', b\in B'$. Then $a-b$ can be written as 
	\begin{equation}
	a-b=(a-b_{1})+(b_{1}-a_{1})+(a_{1}-b)
	\end{equation}
	in at least $\frac{|\bar{A}|^{2}\delta^{2}}{2^{37}}$ ways with all $a-b_{1},a_{1}-b_{1},a_{1}-b\in H$, and so by \eqref{Hbound} 
	\begin{equation}
	\frac{|\bar{A}|^{2}\delta^{2}}{2^{37}}|A'-B'|\leq (16R|\bar{A}|)^{3}
	\end{equation}
	or
	\begin{equation}
	|A'-B'|\leq 2^{49}R^{3}\delta^{-2}|\bar{A}|
	\end{equation}
	By Lemma \ref{RuszaTriangleIneq} we have that
	\begin{equation}
	|A'-A'|\leq 2^{102}R^{6}\delta^{-8}|\bar{A}|
	\end{equation}
	Define
	\begin{equation}
	A_{1}=\left\{a\in A : \exists a'\in A'\quad s.t. \quad \left\lfloor\frac{a}{M}\right\rfloor\cdot M =a'\right\}
	\end{equation}
	And so (using \ref{TruncatedSetSizeIneq})
	\begin{equation}
	\mathcal{N}(A_{1}-A_{1};M)<2^{105}R^{6}\delta^{-8}|A_{0}|
	\end{equation}
\end{proof}

\begin{defn}
	We denote by $V(\overline{y},\rho)$ the $\rho$-neighborhood of $\overline{y}\in\mathbb{P}^1$. Formally it is:
	\begin{equation}
	V(\overline{y},\rho)=\left\{\overline{x}\in\mathbb{P}^1:\frac{|\langle x,y \rangle|}{|x||y|}<\rho\ \text{\ for some\ representatives\ } x,y\in\mathbb{R}^{2}, [x]=\overline{x} , [y]=\overline{y}\right\}.
	\end{equation}
\end{defn}
The following theorem is a projection theorem by Bourgain, which can be found in \cite{MR2763000}.
\begin{thm}[\cite{MR2763000}, Thm. 5]\label{ProjThm} 
	
	 For any $\alpha_{0},\kappa>0$ there are $\alpha_{\Delta}, \epsilon_{0}, r_{0}, \tau_{0}>0$ such that the following holds for $0<r<r_{0}$ and $\alpha_{0}<\alpha<2-\alpha_{0}$: let $\eta$ be a probability measure on $\mathbb{P}^{1}$ s.t.
	\begin{equation}\label{regularity for Bourgain on directions}
	\max_{\overline{y}\in\mathbb{P}^1} \eta\left(V\left(\overline{y},\rho\right) \right)<\rho^{\kappa} \qquad \textnormal{if} \quad r<\rho<r^{\tau_{0}}.
	\end{equation}
	Let $E\subset[0,1]^{2}$ be an r-separated set with $|E|>r^{-\alpha}$ and a non-concentration property
	\begin{equation}\label{regularity for Bourgain on set}
	\max_{x}|E\cap \mathtt{B}_{x,\rho}|<\rho^{\kappa}|E| \qquad \textnormal{if} \quad r<\rho<r^{\tau_{0}}.
	\end{equation}
	Then there exist $D\subset\mathbb{P}^{1}$  and $E'\subset E$ with
	\[
	\eta(D)>1-r^{\epsilon_{0}}, \qquad |E'|>r^{\epsilon_{0}}|E|
	\]
	so that 
	\[
	\mathcal{N}\left( \pi_{\overline{\theta}}(E'');r \right)>r^{-(\alpha+\alpha_{\Delta})/2}
	\]	
	whenever $\overline{\theta}\in D$ and $E''\subset E'$ satisfies $|E''|>r^{2\epsilon_{0}}|E|$. Here, $\pi_{\overline{\theta}}(U)$ is the orthogonal projection of $U$ on the subspace of $R^2$ spanned by some representative of $\overline{\theta}$.
\end{thm}

\end{section}

\begin{section}{Regularity of sets}
	
	\begin{defn}[\cite{MR2726604}, Def. 5.1]\label{BFLM_Def5.1}
		A Borel probability measure $\rho$ on a Borel set $B\subset\mathbb{R}^d$ is said to be $(C,\alpha)$-\emph{regular at scale $r$ on} $B$ if for any $x\in B$, $s\geq r$
		\begin{equation}
		\rho(\mathsf{B}_{x,s})<C\left(\frac{s}{\mathrm{diam}\ B}\right)^{\alpha}.
		\end{equation}
		A set $B$ is said to be $(C,\alpha)$-\emph{regular at scale} $r$ if the corresponding uniform measure $\rho=\frac{1}{|B|}\sum_{x\in B}\delta_{x}$ is $(C,\alpha)$-regular at scale $r$.
	\end{defn}
	
	The following lemma \cite[Lemma 5.2]{MR2726604} relates the defined notion of regularity of a set to the expression of dimension via covering numbers.
	
	\begin{lem}[\cite{MR2726604}, Lemma 5.2]\label{BFLMLemma5.2}
		For any $\epsilon>0$ there are constants $C_{\epsilon},C'_{\epsilon}>0$ such that for every $s,\alpha$ with $2\epsilon<s<\alpha$ and $r<1$, if $\tilde{A}\subset \mathsf{B}_{\mathbf{0},1}\subset \mathbb{R}^d$ satisfies
		\begin{equation}
		\mathcal{N}\left(\tilde{A};r\right)\geq r^{-\alpha},
		\end{equation}
		then there is a point $x\in \mathsf{B}_{\mathbf{0},1}$ and a probability measure $\rho$ supported on $\tilde{A}\cap \mathsf{B}_{x,r^{\beta}}$ which is $\left(C_{\epsilon},\alpha-s\right)$-regular on $\mathsf{B}_{x,C'_{\epsilon}r^\beta}$ at scale $r$ for $\beta=\frac{1-\alpha+\epsilon}{1-\alpha+s-\epsilon}$
	\end{lem}

	Lemma 5.3 in \cite{MR2726604} is the following
	\begin{lem}\label{BFLMLemma5.3}
		Let $\rho$ be a $\left(C,\alpha \right)$-regular probability measure at scale $r$ on $B\subset\mathbb{R}^d$. Then for any $\epsilon>0$ there is an $r$-separated subset $A\subset\mathrm{supp}(\rho)$ and $C_\epsilon>0$ such that the uniform measure on $A$ is $\left(C_{\epsilon},\alpha-\epsilon\right)$-regular at scale $r$ on $B$.
	\end{lem}
	
	We will be using the following lemma. Lemmas \ref{BFLMLemma5.2}, \ref{BFLMLemma5.3} are used in its proof.
	\begin{lem}[{\cite[Lemma 6.7]{MR2726604}, one dimensional torus}]\label{LargeFourierRegularity}
		For any $\epsilon>0$, there is a $C_{\epsilon}$ so that the following holds. Let $\mu$ be a probability measure on $\mathbb{T}$. Assume that for some $N>M,t,\alpha$
		\begin{equation}
		\mathcal{N}(\mathcal{F}(\mu,t)\cap[-N,N];M)\geq\left(\frac{N}{M}\right)^{\alpha}.
		\end{equation}
		Then there is an $M<N_1<N$ with 
		\begin{equation}\label{WindowSizeRegulaityAdjustment}
		\log\frac{N_1}{M}>\left(\frac{1-\alpha+\epsilon}{1-\alpha+8\epsilon} \right)\log\frac{N}{M} 
		\end{equation}
		such that $\mathcal{F}(\mu,t^{2}/4)\cap[-N_1,N_1]$ contains a subset which is $\left(C_{\epsilon}t^{-2},\alpha-10\epsilon\right)$-regular at scale $M$.
	\end{lem}
\end{section}

\begin{section}{Main Bootstrapping Lemma}
The following lemma is very simple but it is employed over and over again in the main lemmas, making its explicit statement and proof worthwhile.
\begin{lem}\label{MarkovLemma}
	Let $\{a_{i}\}_{i=1}^{n}$ be a set of real numbers in $[0,1]$ for which
	\begin{equation}
	\sum_{i=1}^{n}a_{i}\geq \alpha n.
	\end{equation}
	for $\alpha\in[0,1]$. Then
	\begin{equation}
	\left|\left\{i:a_{i}\geq \alpha/2\right\}\right|\geq\frac{\alpha}{2}n
	\end{equation}
	
\end{lem}
\begin{proof}
	Suppose that
	\begin{equation}
	\left|\left\{i:a_{i}\geq \alpha/2\right\}\right|<\frac{\alpha}{2}n
	\end{equation}
	Since each $a_{i}$ can be at most 1, we have that
	\begin{equation}
	\sum_{i=1}^{n}a_{i}<\frac{\alpha}{2}n\cdot 1+\frac{2-\alpha}{2}n\cdot \frac{\alpha}{2}<\alpha n
	\end{equation}
	This is a contradiction.
\end{proof}

We prove the following lemma which is the extraction of the initial set of large Fourier coefficients, using the information of having one single large Fourier coefficient of the random walk measure.
\begin{lem}[Initial dimension]\label{InitialDimensionLemma}
	For any probability measure $\mu$ on $\mathbb{T}$ for $n\geq 1$, if for some $a\in\mathbb{Z}$ 
	\begin{equation}
	|\widehat{\mu}_{n}(a)|>\delta_0
	\end{equation}
	for $\delta_0\in(0,1)$, then  
	\begin{equation}
	\mathcal{N}\left(\mathcal{F}(\mu_{n-1},\frac{\delta_0}{2})\cap[-N,N];M\right)\geq \frac{\delta_{0}}{2}\left(\frac{N}{M}\right)^{\alpha_{ini}}
	\end{equation}
	where $N=L|a|,M=|a|,\alpha_{ini}=\beta$ ($|S|=L^{\beta}$).
\end{lem}
\begin{proof}
	Note the equality
	\begin{equation}
	\widehat{\mu}_{n}(\xi)=\frac{1}{|S|}\sum_{s\in S}\widehat{\mu}_{n-1}(s\xi).
	\end{equation}	
	By the above, we have that
	\begin{equation}
	\delta_{0}<|\widehat{\mu}_{n}(a)|=\left|\frac{1}{|S|}\sum_{s\in S}\widehat{\mu}_{n-1}(sa)\right|\leq\frac{1}{|S|}\sum_{s\in S}\left|\widehat{\mu}_{n-1}(sa)\right|.
	\end{equation}
	By Lemma \ref{MarkovLemma}
	\begin{equation}
	\left|\left\{s\in S: |\widehat{\mu}_{n-1}(sa)|>\frac{\delta_{0}}{2}\right\}\right|\geq \frac{|S|\delta_{0}}{2}
	\end{equation}
	and so
	\begin{equation}
	\mathcal{N}\left(\mathcal{F}(\mu_{n-1},\frac{\delta_0}{2})\cap\left[-L|a|,L|a|\right];|a|\right)\geq |S|\frac{\delta_{0}}{2}
	\end{equation}
\end{proof}

For the next lemma we need the following simple definition.
\begin{defn}
	For a set $A\subset\mathbb{Z}$ and a positive real $M$, let $B_M(A)=\bigcup_{a\in A}\mathsf{B}_{a,M}$.
\end{defn}
The following is the main technical tool of the proof of our main decomposition theorem, Theorem \ref{MainThm1}. We either find a large set of Fourier coefficients by regarding a smaller value of the threshold on the coefficients as being "large", or we look at the previous generation random walk measure; the assumption of non-existence of a set that meets our terms is employed with the additive structure of the Fourier coefficients to show two contradicting inequalities.
\begin{lem} [Bootstrap lemma]\label{BootstrapLemma}
	Given $\lambda,\tau>0$ and $0<\alpha_{ini}<\alpha_{high}<1$, there exist $C^{*},\alpha_{inc}>0$ and $L_1$ such that if $L>L_1$, $\alpha_{high}\geq\alpha\geq\alpha_{ini}$, $S \subset [L,2L] $ is a $(\widetilde{C},\lambda)$-regular set for some $\widetilde C <L^{C^{*}}$, and $N,M,\delta$ satisfy
	\begin{equation}\label{BootstrapNMcond}
	L^{\tau}<\frac{N}{M}<L\quad \text{and} \quad L^{-C^{*}}<\delta,
	\end{equation}
	and if 
	\[
	\mathcal{N}\left(\mathcal{F}(\mu_{n},\delta)\cap[-N,N];M\right)\geq\left(\frac{N}{M}\right)^{\alpha}
	\]
	for some  $n\geq 1$, then
	\[\mathcal{N}\left(\mathcal{F}(\mu_{n-1},\delta^{4}/256)\cap[-N',N'];M'\right) \geq\left(\frac{N'}{M'}\right)^{\alpha+\alpha_{inc}}
	\]
	where $N'=LN_0, M'=LM$, and $N_0$ is such that $M<N_0\leq 2N$ with
	\begin{equation}\label{N1_MtoN_Mrelation}
	\log\frac{N_0}{M}>\frac{1}{8}\cdot \log\frac{N}{M}.
	\end{equation}
\end{lem}
\begin{proof}
	Let $\alpha_{\Delta}$ be as in Theorem \ref{ProjThm} for $\alpha_0=\min(\alpha_{ini},1-\alpha_{high})/2$, and for 	$\kappa=\frac{\lambda}{10}$. 
	
	Let
	\begin{equation}
	E_{0}\subset\mathcal{F}\left(\mu_{n},\delta\right)\cap [-N,N]
	\end{equation}
	be an $M$-separated set of maximal cardinality. Let $\epsilon<\frac{\alpha_{\Delta}}{640\cdot 20}$ be a constant to be
	determined when we explain how to apply the projection theorem (Theorem 17) later
	in this proof. Note that the proof may end without actually applying the projection theorem. Apply Lemma \ref{LargeFourierRegularity} with respect to $\mu_n,\delta$ (in the roles of $\mu,t$) to obtain
	\begin{equation}
	E'_0\subset \mathcal{F}\left(\mu_n,\delta^2/4\right)\cap[-N_1,N_1],
	\end{equation}
	an $M$-separated set which is $\left(C\delta^{-2},\alpha-10\epsilon\right)$-regular at scale $M$ ($C$ depends on $\epsilon$). $N_1$ is as obtained in the conclusion of Lemma \ref{LargeFourierRegularity}.
	Let
	\begin{equation}
	E_{1}\subset\mathcal{F}\left(\mu_{n},\delta^{4}/32\right)\cap [-2N_1,2N_1]
	\end{equation}
	be an $M$-separated set of maximal cardinality. And let
	\begin{equation}
	\rho=\frac{|E_{1}|}{|E'_{0}|}.
	\end{equation}

	We first deal with the case that 
	\begin{equation}\label{RhoBound}
	\rho \geq 1024C\left(\frac{N_1}{M}\right)^{\alpha_{\Delta}/640}\delta^{-6}.
	\end{equation}
	We will use the relation
	\begin{equation}\label{General fact}
	\widehat{\mu}_{n}(\xi)=\frac{1}{|S|}\sum_{s\in S}\widehat{\mu}_{n-1}(s\xi).
	\end{equation}	
	By passing to a subset $E_{1}'\subset E_{1}$ of size $|E_{1}'|> \frac{|E_{1}|}{4}$ we may assume that for some fixed $\theta\in [0,2\pi)$, for all $\xi\in E_{1}'$, $\mathrm{Re}(e^{i\theta}\cdot\widehat{\mu}(\xi))>\frac{\delta^4}{128}$. Therefore,
	\begin{equation}
		\left\vert\frac{1}{|E'_1|}\sum_{\xi\in E'_1}\widehat{\mu}_n (\xi)\right\vert\geq \frac{\delta^4}{128}.
	\end{equation}
	Then by the relation \eqref{General fact} we have the inequality
	\begin{equation}
	\frac{1}{|S||E_{1}'|}\sum_{s\in S}\sum_{\xi\in E_{1}'}|\widehat\mu_{n-1}(s\xi)|\geq \frac{\delta^4}{128}.
	\end{equation}
	In particular there exists $s_{0}\in S$ such that
	\begin{equation}
	\frac{1}{|E_{1}'|}\sum_{\xi\in E'_{1}}|\widehat{\mu}_{n-1}(s_{0}\xi)|\geq\frac{\delta^{4}}{128}.
	\end{equation}
	Let
	\begin{equation}
	E_{2}\subset \mathcal{F}\left(\mu_{n-1},\delta^{4}/256\right)\cap [-2LN_1,2LN_1]
	\end{equation}
	be an $(LM)$-separated set of maximal cardinality.
	By Lemma \ref{MarkovLemma} we have that 
	\begin{equation}\label{BootstrapRhoPartEnd}
	\begin{split}
	\left|E_2\right| & \geq \left|\left\{\xi\in E'_{1}:|\widehat{\mu}_{n-1}(s_{0}\xi)| \geq  \frac{\delta^{4}}{256}\right\}\right|	\\
	& \geq \frac{\delta^{4}}{256}|E'_{1}| \\
	& \geq \rho\frac{\delta^{4}}{1024}|E'_0|.
	\end{split}
	\end{equation}
	If the inequality \eqref{RhoBound} holds then we are done (as long as we choose $\alpha_{inc} \leq \frac{\alpha_\Delta}{1280}$, and for $N_0=2N_1$) as 
	\[
	\left|E_2\right| \leq \mathcal{N}\left(\mathcal{F}(\mu_{n-1},\delta^{4}/256)\cap[-2LN_1,2LN_1];LM\right)	
	\] and $\epsilon<\frac{\alpha_{\Delta}}{640\cdot 20}$.
	
	We now turn to the harder case where \eqref{RhoBound} fails. Define $\delta'=\frac{\delta^2}{4}$.
	By Lemma \ref{FourierBSG} there exists a set $E\subset E'_{0}$ such that
	\begin{enumerate}[label=\textit{(E.\arabic*)}]
		\item $\left|\frac{1}{|E|}\sum_{e\in E}\widehat{\mu}_n (e)\right|\geq \frac{\delta'}{2}$ \label{E1}.
		\item $|E|>c'|E'_{0}|\delta'^{2}$. \label{E2}
		\item $\mathcal{N}(E-E;M)<|E'_{0}|c\rho^{6}\delta'^{-8}$. \label{E3}
	\end{enumerate}	
	Where $c,c'$ are absolute constants.
	Set $\rho'= 1024C\left(\frac{N}{M}\right)^{\alpha_{\Delta}/640}\delta^{-6}$ (this is the value in the bound in \eqref{RhoBound}). 
	By \eqref{General fact} we have the inequality
	\begin{equation}\label{BootstrapCopyStart}
	\left|\frac{1}{|S|} \frac{1}{|E|}\sum_{s\in S}\sum_{e\in E}\widehat{\mu}_{n-1}(se)\right|\geq \frac{\delta'}{2}.
	\end{equation}
	By Lemma \ref{MarkovLemma} we have that
	\begin{equation}\label{FourierSemigroupPairIneq}
	\#\left\{ (s,e)\in S\times E: |\widehat{\mu}_{n-1}(se)|\geq\frac{\delta'}{4}\right\}\geq \frac{\delta'}{4}|S||E|.
	\end{equation}
	Assume that for $\alpha_{inc}$ small to be determined later (but certainly $ \leq \frac{\alpha_\Delta}{1280}$), the following holds:	
	\begin{equation}{\label{AssumeToCont}}
	\mathcal{N}\left(\mathcal{F}(\mu_{n-1},\frac{\delta'}{4})\cap[-LN_1,LN_1];LM\right)<\left(\frac{N_1}{M}\right)^{\alpha+\alpha_{inc}}.
	\end{equation}
	Let
	\begin{equation} 
	M'=LM,\qquad N'=LN_1.
	\end{equation}
	Let 
	\begin{equation}
	E_{3}\subset \mathcal{F}(\mu_{n-1},\frac{\delta'}{4})\cap[-N^{'},N^{'}]
	\end{equation}
	be an $M'$-separated set of maximal cardinality. By inequality \eqref{FourierSemigroupPairIneq} we have that
	\begin{equation}
	\frac{2}{M'}\sum_{s\in S} m\left(B_{M'}(sE)\cap B_{M'}(E_{3})\right) \geq \frac{\delta'}{4}|S||E|.
	\end{equation}
	Therefore,
	\begin{equation}
	\int_{B_{M'}(E_{3})}\sum_{s\in S}\mathbf{1}_{B_{M'}(sE)\cap B_{M'}(E_{3})}(x) dm(x)\geq \frac{\delta'}{8}M'|S||E|.
	\end{equation}
	By the Cauchy-Schwarz inequality we have that
	\begin{equation}\label{AfterCS}
	2M'|E_{3}|	\int_{B_{M'}(E_{3})}\left(\sum_{s\in S}\mathbf{1}_{B_{M'}(sE)\cap B_{M'}(E_{3})}(x)\right)^{2}dm(x) \geq \frac{\delta'^{2}}{64}M'^{2}|S|^{2}|E|^{2} .
	\end{equation}
	Writing inequality \eqref{AfterCS} in the following way
	\begin{equation}
	\frac{1}{|S|^{2}}\sum_{(s_{1},s_{2})\in S\times S}\frac{m\left(B_{M'}(s_{1}E)\cap B_{M'}(s_{2}E)\right)}{M'|E|}\geq\frac{\delta'^{2}|E|}{128|E_3|},
	\end{equation}
	we see (using Lemma \ref{MarkovLemma}) that
	\begin{equation}
	\#\left\{(s_{1},s_{2})\in S\times S:m(B_{M'}(s_{1}E)\cap B_{M'}(s_{2}E))\geq\frac{M'\delta'^{2}|E|^{2}}{256|E_3|}\right\}\geq\frac{\delta'^{2}|E|}{256|E_3|}|S|^{2}.
	\end{equation}
	For a specific $s_1\in S$ we define the set $B$ as follows:
	\begin{equation}\label{defSetB}
	B=\left\{s_{2}\in S:m\left(B_{M'}(s_{1}E)\cap B_{M'}(s_{2}E)\right)\geq\frac{M'\delta'^{2}|E|^{2}}{256|E_3|}\right\}.
	\end{equation}
	By the pigeonhole principle there exists $s_1\in S$ so that
	\begin{equation}
	|B|\geq \frac{\delta'^{2}|E|}{512|E_3|}|S|.
	\end{equation}
	By assumption $S$ is $(\widetilde{C},\lambda)$-regular, and $|E_3|< \left(\frac{N_1}{M}\right)^{\alpha+\alpha_{inc}}$. By \ref{E2} and the $(C\delta^{-2}, \alpha-10\epsilon)$ regularity of $E_0'$ we have
	\begin{equation}\label{size of E}
	|E|>C^{-1}\delta^{4}\left(\frac{N_{1}}{M}\right)^{\alpha-10\epsilon}
	\end{equation}	
	(recall that $\delta'=\delta^2/4$), Hence, we may conclude that
    \[
	\frac{|E|}{|E_3|}\geq\frac{c''\delta^6(N_1/M)^{\alpha-10\epsilon}}{(N_1/M)^{\alpha+\alpha_{inc}}},
    \]
	 from which it follows that $B$ is $(\widetilde{C_1},\lambda)$-regular, where 
	\[
		\widetilde{C_1}=C''L^{10C_*}(N_1/M)^{\alpha_{inc}+10\epsilon}\widetilde{C}
	\]
	for some $c'',C''$ independent of $N,M,L$.
	
	Using Rusza's triangle inequality (Lemma \ref{RuszaTriangleIneq}) we have the following, for fixed $s_1$ and for all $s_2\in B$:
	\begin{multline*}
	\begin{aligned}
	\left|B_{M'}(s_{1}E)-B_{M'}(s_{2}E) \right|\leq&\frac{\left|B_{M'}(s_{1}E)-B_{M'}(s_{1}E)\cap B_{M'}(s_{2}E)\right|\left|B_{M'}(s_{2}E)-B_{M'}(s_{1}E)\cap B_{M'}(s_{2}E) \right|}{\left|B_{M'}(s_{1}E)\cap B_{M'}(s_{2}E)\right|}  \\
	\leq&\frac{\left|B_{M'}(s_{1}E)-B_{M'}(s_{1}E)\right|\left|B_{M'}(s_{2}E)-B_{M'}(s_{2}E)\right|}{\left|B_{M'}(s_{1}E)\cap B_{M'}(s_{2}E)\right|} \\
	\leq& \frac{25\mathcal{N}(E-E;M)^{2}M'^{2}}{|B_{M'}(s_{1}E)\cap B_{M'}(s_{2}E)|}.
	\end{aligned}
	\end{multline*}	
	This can be summarized by the inequality
	\begin{equation}
	\left|B_{M'}(s_{1}E)-B_{M'}(s_{2}E) \right|\leq  \frac{25\mathcal{N}(E-E;M)^{2}M'^{2}}{|B_{M'}(s_{1}E)\cap B_{M'}(s_{2}E)|}
	\end{equation}
	By \ref{E2}, \ref{E3} and by \eqref{defSetB},\eqref{AssumeToCont}, for $s_2\in B$ ($s_1$ fixed before), we have
	\begin{equation}\label{DiffLowerBound}
	\begin{split}
	\left|B_{M'}(s_{1}E)-B_{M'}(s_{2}E) \right| & \leq c\rho^{12}\delta'^{-22}M'|E_3|\\
	& < c_1\rho^{12}\delta^{-44}M'\left(\frac{N_1}{M}\right)^{\alpha+\alpha_{inc}}\\
	& < c_2 \delta^{-116}M' \left(\frac{N_1}{M}\right)^{\alpha+\alpha_{inc+3\alpha_\Delta/16}}.
	\end{split}
	\end{equation}
	


		Define the set $\widetilde{E}=N_{1} ^{-1}E\times N_{1}^{-1}E$. By \eqref{size of E},	
	\begin{equation}
	|\widetilde{E}|>C^{-2}\delta^{8}\left(\frac{N_{1}}{M}\right)^{2\alpha-20\epsilon}>\left(\frac{N_{1}}{M}\right)^{2\alpha-21\epsilon}
	\end{equation}
	by \eqref{BootstrapNMcond} and~\eqref{N1_MtoN_Mrelation}, the second inequality holding if $C^{*} < 2^{-6} \tau \epsilon$ and $L_1$ is large enough.

	The set $\widetilde{E}$ is $\frac{M}{N_{1}}$-separated and $(C^{2}\delta^{-8},2\alpha-20\epsilon)$-regular at scale $M/N_{1}$. 
	We apply the projection theorem, Theorem~\ref{ProjThm}, to the set $\widetilde{E} \subset [-1,1]^2$ with respect to the measure $\eta$ on the set of directions in $\mathbb{P}^{1}$ corresponding to uniform choice of direction from the projection of the set $\{-s_1\}\times B$ to $\mathbb{P}^{1}$. This measure $\eta$ will satisfy \eqref{regularity for Bourgain on directions} for any $\kappa<\lambda$ as long as the $\tau_0$ from Theorem~\ref{ProjThm} satisfies that $\left(\frac{N_1}{M}\right)^{\tau_0(\lambda-\kappa)} > \widetilde{C_1}$, which holds for suitable choice of $C^{*},\epsilon,\alpha_{inc}$ once $L_1$ is large enough.
	Similarly, $\widetilde{E}$ will satisfy \eqref{regularity for Bourgain on set} once $\kappa < \alpha_{ini}$ if $C^{*},\epsilon$ are small enough and $L_1$ large enough.

	Theorem \ref{ProjThm} gives us a set 
	$\Theta\subset \mathbb{P}^{1}$  with $\eta (\Theta)>1-\left(\frac{M}{N_1}\right)^{\epsilon_0}$ 
	so that for $\overline{(-s_1,s_2)}\in\Theta$:
	\begin{equation}\label{ContradictionIneqUB}
	\begin{split}
	\mathcal{N} & \left(\frac{s_{2}E/N_{1}-s_{1}E/N_{1}}{\sqrt{s_{1}^{2}+s_{2}^{2}}};\frac{M}{N_{1}}\right)\geq \left(\frac{N_1}{M}\right)^{\alpha+\alpha_{\Delta}-11\epsilon} \\
	\end{split}
	\end{equation}
	Since $\eta(\Theta)$ is positive, there is at least one $s_2\in B\subset S$, for which the inequality \ref{ContradictionIneqUB} holds. Let $s_2$ be any such number.
	
	Using the fact that $\sqrt{s_1^2+s_2^2}\geq L$ and by $M'=LM$, we have (if $\epsilon$ is $<\alpha_\Delta/100$) the inequality
	\begin{equation}
	\mathcal{N}\left(s_{2}E-s_{1}E;M'\right)\geq \left(\frac{N_1}{M}\right)^{\alpha+\alpha_{\Delta}/2}.
	\end{equation}
	Therefore,
	\begin{equation}\label{DiffUpperBound}
	\left|B_{M'}(s_{1}E)-B_{M'}(s_{2}E) \right|\geq 25\mathcal{N}\left(s_{2}E-s_{1}E;M'\right)M'\geq M'\left(\frac{N_1}{M}\right)^{\alpha+\alpha_{\Delta}/2}.
	\end{equation}
	
	Recalling that $\alpha_{inc}$ was chosen to be $\leq\alpha_{\Delta}/1280$, we get a contradiction between \eqref{DiffLowerBound} and \eqref{DiffUpperBound} if $C^{*}$ is small enough and $L_1$ large enough. This completes the proof (with $N_0=N_1$).
	

\end{proof}

\end{section}

\begin{section}{Dimensions of Projections}
	This section contains background material for a final bootstrapping lemma, which is stated and proved the end.
	The following part is adapted from \cite{MR2726604}. Closely related to the notion of $(C,\alpha)$-regular measure introduced in Definition \ref{BFLM_Def5.1} is the notion of $\alpha$-energy of a measure $\rho$, denoted by $\mathcal{E}_\alpha (\rho)$, which we define for a compactly supported measure $\rho$ on $\mathbb{R}^d$ and $\alpha<d$.
	\begin{defn}
		The $\alpha$-\emph{energy} of a compactly supported measure $\rho$ on $\mathbb{R}^d$ and $\alpha<d$, denoted by $\mathcal{E}_\alpha (\rho)$, is defined by
		\begin{equation}
		\mathcal{E}_\alpha (\rho)=\int_{\mathbb{R}^d}\int_{\mathbb{R}^d}\frac{d\rho(x)d\rho(y)}{|x-y|^\alpha}.
		\end{equation}
	\end{defn}
	
	If $\rho$ is $(C,\alpha+\epsilon)$-regular on a set $B$ at all scales, then
	\begin{equation}\label{EnergyAsym}
	\mathcal{E}_\alpha (\rho)= \alpha\int\int \frac{\rho(\mathsf{B}_{x,r})}{r^{\alpha+1}}d\rho(x)dr\leq C(\mathrm{diam}\ B)^{-\alpha}\alpha\epsilon^{-1}.
	\end{equation}
	
	The energy $\mathcal{E}_\alpha (\rho)$ can also be expressed in terms of the Fourier transform of $\rho$, up to an implicit constant that tends to $\infty$ as $\alpha\rightarrow 1$ (see \cite{mattila1999geometry}, Lemma 12.12):
	\begin{equation}\label{EnergyAsym2}
	\mathcal{E}_{\alpha}(\rho)\asymp\int_{\mathbb{R}^d}|\widehat{\rho}(\xi)|^2(1+|\xi|)^{\alpha-d}d\xi.
	\end{equation}	
	
	If $\mathcal{E}_\alpha (\rho)<\infty$, then any set of positive $\rho$ measure has Hausdorff dimension at least $\alpha$ (for this and further information about $\alpha$-energy, see \cite{mattila1999geometry}).
	
	A simple way to adapt this notion to our "coarse" setup, where we do not care about the details of how $\rho$ behaves at scales smaller than $r$, is to smooth it by convolving with an appropriate kernel. Let $\Phi$ be a fixed radially symmetric nonnegative smooth function on $\mathbb{R}^d$ with $\left\Vert\Phi\right\Vert_1=1$ supported on $\mathsf{B}_{\mathbf{0},1}$ and for $r>0$ set 
	\begin{equation}\label{normalize_phi}
	\Phi_r (x)=r^{-d}\Phi(r^{-1}x).
	\end{equation}
	Then instead of using the possibly atomic measure $\rho$, we can consider its smoothed version $\rho'=\rho*\Phi_r$. In particular, if $\rho$ is $(C,\alpha+\epsilon)$-regular at scale $r$ on a subset $B\subset\mathbb{R}^d$, then
	\begin{equation}
	\mathcal{E}_\alpha (\rho *\Phi_r)\ll C(\mathrm{diam}\ B)^{-\alpha}\alpha\epsilon^{-1}
	\end{equation}
	with the implicit parameter depending only on $d$ and the choice of $\Phi$.

	See \cite{MR2726604}, subsection 6.C. for more details.
	Let $\Psi:\mathbb{R}\rightarrow\mathbb{R}^+$ be the smooth compactly supported function
	\begin{equation}
	\Psi(x_1)=\int dx_2 ...\int dx_d\Phi(x_1,x_2,...,x_d), 
	\end{equation}
	and define $\Psi_r$ analogously to (\ref{normalize_phi})
	
	\begin{lem}[\cite{MR2726604}, Lemma 6.10]\label{BFLM610}
		Let $\rho$ be a probability measure on $\mathbb{R}$, and let $\phi$ be the Radon-Nikodym derivative $\phi=\frac{d(\rho*\Psi_r)}{dx}$. Then for every $0<r<r_1<1$
		\begin{equation}
		\mathcal{N}(\mathrm{supp}\ \rho;r_1)\geq (4r_1\left\Vert\phi\right\Vert_2^2)^{-1}.
		\end{equation}
		Moreover, for any subset $X\subset\mathrm{supp}\ \rho$,
		\begin{equation}
		\mathcal{N}(X;r_1)\geq\frac{\rho(X)^2}{4r_1\left\Vert\phi\right\Vert^2_2}.
		\end{equation}
	\end{lem}
	For the next proposition we need a further definition.
	\begin{defn}
		Let $\rho$ be a probability measure supported on the unit ball $\mathsf{B}_{0,1}$ of $\mathbb{R}^d$ and let $\rho_{\theta}$ be the orthogonal projection of the measure $\rho$ in the direction $\theta\in \mathbb{P}^{d-1}$. Then $\widehat{\rho_{\theta}}(t)$ is defined by
		\begin{equation}
			\widehat{\rho_{\theta}}(t)=\int_{-\infty}^{\infty}e^{-2\pi itz}d\rho_{\theta}(z).
		\end{equation}
	\end{defn}
	\begin{prop}[\cite{MR2726604}, Prop. 6.11]\label{BFLM611}
		Let $\rho$ be a probability measure supported on the unit ball $\mathsf{B}_{0,1}$ of $\mathbb{R}^d$ so that $\mathcal{E}_{\alpha}(\rho)<\infty$ for some $0<\alpha<d$, $0<r<1$, and let $\eta$ be a measure on $S^{d-1}$ such that for some $c_\eta,\beta>0$
		\begin{equation}\label{PSDirectionReg}
		\eta(\mathsf{B}_{\theta,\epsilon})\leq c_\eta\epsilon^{\beta}\qquad\mathnormal{for\ every\ }\epsilon>r\ and\ \theta\in S^{d-1}.
		\end{equation}
		Then for any $\beta'<\beta$
		\begin{equation}\label{PSDirectional}
		\begin{split}
		\int_{\theta}\int_t |\widehat{\rho_\theta} & (t)|^2\left\vert\widehat{\Psi_r}(t)\right\vert^2(1+|t|)^{\beta'+\alpha-d}dtd\eta (\theta) \\
		& \leq c_\eta C_d \int_{\mathbb{R}^d}|\widehat{\rho}(x)|^2 \left\vert\widehat{\Phi_r} (x)\right\vert^2 (1+|x|)^{\alpha-d}dx+c_\eta C(\alpha,\beta,\beta',d).
		\end{split}
		\end{equation}
	\end{prop}
	We shall use Proposition \ref{BFLM611} with $d=2$. Almost quoting from \cite{MR2726604}, note that if $\alpha+\beta'>d$ and $\rho$ is $(C,\alpha')$-regular at scale $r$ for $\alpha'>\alpha$, then by \eqref{EnergyAsym} the right-hand side of \eqref{PSDirectional} is bounded from above by a constant (depending on $\alpha,\alpha',\beta,\beta',C,...$) while the left hand side is at least
	\begin{equation}
	\int_{\theta}\left\Vert\frac{d(\rho_{\theta} * \Psi_r)}{dx}\right\Vert^2_2 d\eta(\theta).
	\end{equation}
	In view of Lemma \ref{BFLM610}, this in particular implies that for $\eta$-many choices of $\theta$, the covering number of supp($\rho_\theta$) by $r$-intervals is large.

	The next lemma will be used as a final step after the application of a number of iterations of Lemma \ref{BootstrapLemma}.
	\begin{lem}[High dimension to positive density]\label{FinalBootstrap}
		For any $\lambda>0$ there exist $\epsilon_0>0$ and $L_1, C^*>0$ such that if $L>L_1$, $\delta>L^{-C^*}$, $S \subset [L,2L]$ is $(\widetilde{C},\lambda)$-regular, $N\leq LM$ and
		\begin{equation}\label{AlmostHDIneq}
		\mathcal{N}\left(\mathcal{F}(\mu_n,\delta)\cap[-N,N];M\right)>\left(\frac{N}{M}\right)^{1-\epsilon_0},
		\end{equation}
		then there exists $N_1$ such that
		\begin{equation}
		\mathcal{N}\left(\mathcal{F}(\mu_{n-1},\delta^4/128 )\cap[-N_1,N_1];M\right)>\frac{c\delta^{10}}{\widetilde{C}}\frac{N_1}{M} ,
		\end{equation}
		where $c$ is a constant and $N_1$ is such that
		\begin{equation}
		\log\frac{N_1}{M}>\frac{1}{2}\log\frac{N}{M}.
		\end{equation}
	\end{lem}
	\begin{proof}
		(Parts of this Lemma are adapted from \cite[Lemma 6.12]{MR2726604}).
		Set $\epsilon_0=\lambda/60$. 
		Assume that for $N,M,n,\delta$ the inequality \ref{AlmostHDIneq} holds. Let
		\begin{equation}
		E_{0}\subset\mathcal{F}\left(\mu_{n},\delta\right)\cap [-N,N]
		\end{equation}
		be an $M$-separated set of maximal cardinality.

		Set $\delta'=\delta^2/4$. By Lemma \ref{LargeFourierRegularity} applied with $\epsilon = \lambda/6$, there exists $N_1\in(M,N)$ with $\log(N_1/M)>\frac{1}{2}\log(N/M)$ such that $\mathcal{F}(\mu_n,\delta')\cap[-N_1,N_1]$ contains a subset $E$ which is $(C\delta'^{-2},1-2\lambda/6)$-regular at scale $M$, where $C$ depends only on $\lambda$. We may assume that
		\begin{equation}
		\frac{1}{E}\left\vert\sum_{b\in E}\widehat{\mu}_n(b)\right\vert\geq \frac{\delta'}{2},
		\end{equation}
		since we may always choose a subset $E_1\subset E$ of cardinality $\geq$ $|E|/4$ on which the above inequality holds which 
		is $(C\delta'^{-2},1-2\lambda/6)$-regular (possibly for a slightly different $C$).
		
		 Set $\phi(x)=\sum_{s\in S}\sum_{\xi\in E}e_{s\xi}(x)$. Then by the Cauchy-Schwarz inequality we have
		 \begin{equation}
		 \begin{split}
			 \sum_{s_1,s_2\in S}\sum_{\xi_1,\xi_2\in E}\widehat{\mu}_{n-1}(s_1\xi_1-s_2\xi_2) &= \int_{\mathbb{T}}\sum_{s_1,s_2\in S}\sum_{\xi_1,\xi_2\in E} e_{s_1\xi_1-s_2\xi_2}(x)d\mu_{n-1}(x) \\
			 & = \int_{\mathbb{T}}|\phi(x)|^2d\mu_{n-1}(x) \\
			 & \geq \left\vert \int_{\mathbb{T}}\phi(x) d\mu_{n-1}(x)\right\vert^2\\
			 & =\left\vert\sum_{s\in S} \sum_{\xi\in E} \widehat{\mu}_{n-1}(s\xi)\right\vert^2\\
			 & = \left\vert\ |S| \sum_{\xi\in E} \widehat{\mu}_{n}(\xi)\right\vert^2.
		\end{split}
		 \end{equation} 
		We then obtain,
		\begin{equation}
		\frac{1}{|S|^2}\sum_{s_1,s_2\in S}\sum_{\xi_1,\xi_2\in E}\widehat{\mu}_{n-1}(s_1\xi_1-s_2\xi_2)\geq \left\vert\sum_{\xi\in E} \widehat{\mu}_{n}(\xi)\right\vert^2,
		\end{equation}
		and so
		\begin{equation}
		\begin{split}
		\frac{1}{|E|^2}\frac{1}{|S|^2}\sum_{s_1,s_2\in S}\sum_{\xi_1,\xi_2\in E}\widehat{\mu}_{n-1}(s_1\xi_1-s_2\xi_2)&\geq \left\vert\frac{1}{E}\sum_{\xi\in E} \widehat{\mu}_{n}(\xi)\right\vert^2\\
		&\geq \frac{\delta'^2}{4}.
		\end{split}
		\end{equation}
		
		Fix $s_2$ to be an element in $S$ such that the term corresponding to it in the above sum is the largest. Then
		\begin{equation}
		\frac{1}{|S|}\frac{1}{|E|^2}\sum_{s_1,s_2\in S}\sum_{\xi_1,\xi_2\in E}\widehat{\mu}_{n-1}(s_1\xi_1-s_2\xi_2)\geq \frac{\delta'^2}{4}.
		\end{equation}
		By Lemma \ref{MarkovLemma} we have that
		\begin{equation}\label{PDNumberIneq}
			\#\left\{(s_1,\xi_1,\xi_2)\in S\times E\times E: |\widehat{\mu}_{n-1}(s_1\xi_1-s_2\xi_2)|\geq \frac{\delta'^2}{8}\right\}\geq\frac{\delta'^2}{8}|S||E|^2.
		\end{equation}
		Let 
		\begin{equation}
			Q=\left\{(s_1,\xi_1,\xi_2)\in S\times E\times E: |\widehat{\mu}_{n-1}(s_1\xi_1-s_2\xi_2)|\geq \frac{\delta'^2}{8}\right\}.
		\end{equation}
		Next, we define a set $S'$ by
		\begin{equation}
			S'=\left\{s_1\in S: \left|Q\cap \left(\{s_1\}\times E \times E\right)\right|\geq\frac{\delta'^2|E|^2}{16} \right\}.
		\end{equation}
		By Lemma \ref{MarkovLemma} we have that $|S'|\geq \frac{\delta'^2}{16}|S|$. 
		
		Let $\eta$ be the uniform measure on the set of directions in $\mathbb{P}^1$ corresponding to the set $\{-s_2\}\times S'$. The $(\widetilde{C},\lambda)$-regularity of $S$ ensures that for any $\xi\in\mathrm{supp}(\eta)\subset\mathbb{P}^1$ we have the inequality
		\begin{equation}
		\eta\left(V(\xi,r)\right)\leq u\widetilde{C}\delta'^{-2}r^{\lambda}
		\end{equation}
		for any positive real number $r\geq M/N$ and some absolute constant $u$. 
		Applying Proposition \ref{BFLM611} with $\beta=\lambda$, $\beta'=\frac{5\lambda}{6}$, $\alpha = 1-\frac{5\lambda}{6}$ and \[\rho=\frac{1}{|E|^2}\sum_{b\in E''\times E'}\delta_{b/N_1}\] we get that
		\begin{equation}\label{Prop24output}
		\int_\xi \left\Vert\frac{d(\rho_\xi *\Phi_r)}{dx}\right\Vert^2_2d\eta\leq \widetilde{C}\delta'^{-2}\left[C_2\int_{\mathbb{R}^2}|\widehat{\rho}(x)|^2\left\vert\widehat{\Phi_r}(x)\right\vert^2(1+|x|)^{\alpha-2}dx+C(\alpha,\beta,\beta')\right],
		\end{equation}
		for $r=M/N_1$. 
		Recall that $\rho$ is $(C\delta^{-2},2-2\lambda/3)$-regular at scale $M/N_1$ (by adjusting $C$ if needed).
		It follows that
		\begin{equation}
		\begin{split}
		\int_{\mathbb{R}^2}|\widehat{\rho}(x)|^2\left\vert\widehat{\Phi_r}(x)\right\vert^2 (1+|x|)^{\alpha-2}dx & \asymp\mathcal{E}_{\alpha}(\rho*\Phi_r)\leq \qquad\qquad\ (\mathrm{by}\ \ref{EnergyAsym2})\\
		& \leq c''\delta^{-2}=8c''\delta'^{-1}\qquad (\mathrm{since}\ \alpha<2-2\lambda/3)
		\end{split}		
		\end{equation}
		with $C',c''$ depending on $\lambda$. Substituting into \ref{Prop24output}, we get
		\begin{equation}
		\int_{\xi}\left\Vert\frac{d(\rho_{\xi}*\Phi_r)}{dx}\right\Vert^2_2d\eta (\xi)\leq \widetilde{C}c_*\delta'^{-3}.
		\end{equation}
		We conclude that there is a subset $S''\subset S'$ with $|S'|>(1-\frac{\delta'^2}{16})|S|$ for which if $s_1\in S''$ and $\xi_0=\overline{(-s_2,s_1)}\in\mathbb{P}^1$, then 
		\begin{equation}\label{RadonNikodynBound}
		\left\Vert\frac{d(\rho_{\xi_0}*\Phi_r)}{dx}\right\Vert^2_2\leq \widetilde{C}c_*\delta'^{-3}\cdot\frac{16}{\delta'^2}.
		\end{equation}
		For any such direction $\xi_0\in\mathbb{P}^1$, let $\pi_{\xi_0}$ denote the orthogonal projection on to the subspace spanned by $\xi_0$ (considered as a map $\mathbb{R}^2\rightarrow\mathbb{R}$). By Lemma \ref{BFLM610} and \ref{RadonNikodynBound} it follows that
		\begin{equation}\label{HDContradictionUB}
		\mathcal{N}\left(\pi_{\xi_0}(\frac{E''\times E'}{N_1}; \frac{M}{N_1})\right)\geq 4\delta'^3\cdot\frac{\delta'^2}{16\widetilde{C}c_*}\cdot \frac{N_1}{M}\ .
		\end{equation}
		This yields the conclusion of our lemma.
	\end{proof}

\end{section}

\begin{section}{The Main Granulation Estimate}
	
We state and prove two key propositions . The first is a general statement which is stated and proved in \cite{MR2726604}. The second is the main granulation estimate, which is used in the proof of the main theorem, Theorem \ref{MainThm1}. 

The following proposition and its proof are adapted from Bourgain, Furman, Lindenstrauss and Mozes, \cite{MR2726604}.
The statement and its proof are harmonic analytic in nature.
	\begin{prop}[{\cite[Proposition 7.5]{MR2726604}}]\label{GranulationProp}
		There exists $c>0$ such that if $t>0$ and a probability measure $\mu$ on $\mathbb{T}^d$ satisfies 
		\begin{equation}
		\mathcal{N}\left(\left\lbrace a\in \mathbb{Z}^d\cap \mathsf{B}_{0,N} : |\widehat{\mu}(a)|>t\right\rbrace ; M\right)>s\cdot\left(\frac{N}{M}\right)^d
		\end{equation}
		with $M<\mathrm{const}_d \cdot N$, then there exists an $\frac{1}{M}$-separated set $X\subset \mathbb{T}$ with
		\begin{equation}
		\mu\left( \bigcup_{x\in X} \mathsf{B}_{x,\frac{1}{N}}\right)>c\cdot(ts)^3.
		\end{equation}
	\end{prop}
\begin{proof}
	We shall need an auxiliary smooth function $F$ on the torus such that
	\begin{equation}
	0\leq F \leq C_{1}\cdot N^d,\qquad \mathrm{supp}(F)\subset\mathsf{B}_{0,\frac{1}{N}}\qquad\int_{\mathbb{T}^d}Fdx=1
	\end{equation}
	and 
	\begin{equation}
	\widehat{F}(a)\geq 0,\qquad\widehat{F}(a )\geq\frac{1}{2}\quad\mathrm{for}\quad a\in\mathbb{Z}^d\cap \mathsf{B}_{0,N}
	\end{equation}
	where $C_1$ is a constant depending on $d$ only. 
	To construct such a function, consider the step function $F_{1}(x)=m\left(\mathsf{B}_{0,r}^{-1}\right)\cdot \mathbf{1}_{\mathsf{B}_{0,r}}(x)$, where $r=\epsilon/N$ for some fixed small $\epsilon>0$. Then $\widehat{F_1}(a)$ is close to 1 for $a\in\mathbb{Z}^d\cap\mathsf{B}_{0,N}$. If $F_2$ is a smooth symmetric approximation of $F_1$, then the convolution $F=F_{2}*\check{F_2}$ has the desired properties.
	
	Let $\tilde{A}$ be an $M$-separated set of size $|\tilde{A}|>s(N/M)^d$ consisting of coefficients $a\in\mathbb{Z}^d\cap\mathsf{B}_{0,N} $ with $|\widehat{\mu}(a)|>t$. Upon passing to a subset $A\subset\tilde{A}$ of size
	\begin{equation}
	|A|\geq\frac{|\tilde{A}|}{4}>\frac{s}{4}\left(\frac{N}{M}\right)^d,
	\end{equation}
	we may assume that $\mathrm{Re}(e^{i\theta}\cdot \widehat\mu(a))>\frac{t}{2}$ for some fixed $\theta\in [0,2\pi]$. Let
	\begin{equation}
	\phi(x)=\sum_{a\in A}e_{a}(x). 
	\end{equation}
	As usual, $e_{a}=e^{-2\pi i  xa}$. Note that
	\begin{equation}
	|\phi(x)|^{2}=(\sum_{a\in A}e_{a}(x))\cdot\overline{(\sum_{b\in A}e_{b}(x))}=\sum_{a,b\in A}e_{a-b}(x).
	\end{equation}
	The probability measure $\lambda=\mu*F$ has a smooth density $g:\mathbb{T}^d\rightarrow[0,\infty)$ with $\widehat{g}(b)=\widehat{\mu}(b)\cdot\widehat{F}(b)$. On $A$ we have $\widehat{F}\geq 1/2$ and $\mathrm{Re}(e^{i\theta} \widehat\mu(a))>t/2$. Therefore
	\begin{equation}\label{HA_ref1}
	\left|\int_{\mathbb{T}^d}\phi d\lambda\right|\geq\sum_{a\in A}\mathrm{Re}(e^{i\theta}\cdot\widehat{g}(a))>\frac{t}{4}\cdot|A|>\frac{ts}{2^{4}}\cdot\left(\frac{N}{M}\right)^d .
	\end{equation}
	We shall see that the right-hand side is close to an a priori upper estimate for the left-hand side. Partition $\mathbb{T}^d$ into $M^d$ "cubes" $Q_i$ with side length $\frac{1}{M}$ and centers $c_i\in\mathbb{T}^d$. By the Cauchy-Schwarz inequality,
	\begin{equation}\label{HA_ref2}
	\left|\int_{\mathbb{T}^d}\phi d\lambda\right|\leq\sum_{i}\left|\int_{\mathbb{T}^d}\mathbf{1}_{Q_i}\cdot\phi d\lambda\right|\leq\sum_ {i}\lambda(Q_i)^{\frac{1}{2}}\cdot\left(\int_{Q_i}|\phi|^{2}d\lambda\right)^\frac{1}{2}
	\end{equation}
	Let $r=\frac{1}{M}$ which is assumed to be at least $\frac{1}{N}$. Then $Q_{i}\subset \mathsf{B}_{c_{i},r/2}$ and $y+Q_{i}\subset\mathsf{B}_{c_{i},r}$ for any $y\in\mathrm{supp}(F)\subset\mathsf{B}_{0,\frac{1}{N}}$. Thus,
	\begin{equation}
	\lambda (Q_i)=\int_{\mathbb{T}^d}F(y)\cdot\mu(y+Q_{i})dy\leq\mu(\mathsf{B}_{c_{i},r}).
	\end{equation}
	Since $d\lambda(x)=g(x)dx$, we have
	\begin{equation}
	\int_{Q_i}|\phi|^{2}d\lambda\leq G_{i}\cdot\int_{Q_{i}}|\phi|^{2}dx,\qquad\mathrm{where}\qquad G_{i}=\max_{x\in Q_{i}}g(x).
	\end{equation}
	We shall estimate $\int_{Q_{i}}|\phi|^{2}dx$ using an auxiliary function $f$ on $\mathbb{T}^d$: we take $f$ to be the product $f(x)=\prod_{i=1}^{d}h_{M}(x_i)$ of one-dimensional Fej\'{e}r kernels
	\begin{equation}
	h_n(u)=\frac{1}{n}\sum_{k=1}^{n}\sum_{j=-k}^{k}e^{2\pi jui}=\frac{1}{n}\left(\frac{\sin\frac{nu}{2}}{\sin\frac{u}{2}}\right)^2 .
	\end{equation}
	Note that $f$ is a nonnegative function, with $f(x)>10^{-d}\cdot M^d$ on the $\frac{1}{M}$-cube $Q_{0}=\left[-\frac{1}{2M},\frac{1}{2M}\right]^d+\mathbb{Z}^d$ around $0\in\mathbb{T}^d$. The Fourier coefficients $\widehat{f}$ take values in $[0,1]$ and vanish outside the $[-M,M]^d\cap\mathbb{Z}^d$ cube. Thus
	\begin{equation}
	\begin{split}
	\int_{Q_{i}} & |\phi(x)|^{2}dx=\int_{Q_{0}}|\phi(x_{i}+y)|^{2}dy\leq\frac{10^d}{M^d}\int_{Q_{0}}|\phi(c_{i}+y)|^{2}f(y)dy \\
	& \leq\frac{10^d}{M^d}\int_{\mathbb{T}^d}|\phi(c_{i}+y)|^{2}f(y)dy=\frac{10^d}{M^d}\int_{\mathbb{T}^d}\sum_{a,b\in A}e_{a-b}(c_{i}+y)\cdot f(y)dy \\
	&=\frac{10^d}{M^d}\left(\sum_{a,b\in A}e_{a-b}(c_{i})\widehat{f}(a-b)\right)\leq\frac{10^d}{M^d}\cdot\sum_{a,b\in A}|\widehat{f}(a-b)|.
	\end{split}
	\end{equation}
	Let $C_2$ denote the constant which is $10^d$ times the maximal cardinality of a $1$-separated set in $[-1,1]^{-d}$. Since $A$ is $M$-separated and $0\leq\widehat{f}\leq 1$, we have
	\begin{equation}
	\frac{10^d}{M^d}\cdot\sum_{a,b\in A}|\widehat{f}(a-b)|\leq\frac{C_2\cdot|A|}{M^d}\leq\frac{C_2\cdot N^d}{M^{2d}}.
	\end{equation}
	The density $g$ of $\lambda=\mu*F$ has the following upper bound:
	\begin{equation}\label{HA_ref3}
	g(x)=\int_{\mathbb{T}^d}F(x-y)d\mu(y)\leq C_{1}\cdot N^d \cdot \mu(\mathsf{B}_{x,\frac{1}{N}}).
	\end{equation}
	Since $\mathsf{Nbd}_{\frac{1}{N}}(Q_{i})\subset\mathsf{B}_{c_{i},r}$, it follows that
	\begin{equation}
	G_{i}=\max_{x\in Q_{i}} g(x)\leq C_{1}N^d\mu(\mathsf{B}_{c_{i},r}).
	\end{equation}
	Let $0\leq H_{i}\leq 1$ denote the ratio, so $G_{i}=H_{i}\cdot C_{1}N^d\mu(\mathsf{B}_{c_{i},r})$. By \ref{HA_ref1} and \ref{HA_ref2},
	\begin{equation}
	\begin{split}
	\frac{ts}{2^{4}}\left(\frac{N}{M}\right)^d &\leq \sum_{i}\mu(\mathsf{B}_{c_{i},r})^{\frac{1}{2}}\cdot G_{i}^{\frac{1}{2}}\cdot\frac{\sqrt{C_2}N^\frac{d}{2}}{M^d}\\
	& \leq\sum_{i}\mu(\mathsf{B}_{c_{i},r})\cdot H_{i}^{\frac{1}{2}}\cdot\sqrt{C_{1}\cdot C_{2}}\cdot\left(\frac{N}{M}\right)^d.
	\end{split}
	\end{equation}
	Let $C_{3}=\sqrt{C_1\cdot C_2}$. We have
	\begin{equation}
	\sum_{i}\mu(\mathsf{B}_{c_{i},r})\cdot H_{i}^{\frac{1}{2}}>\frac{ts}{2^{4}C_{3}}.
	\end{equation}
	Therefore,
	\begin{equation}\label{HA_ref4} 
	\sum_{i\in I}\mu(\mathsf{B}_{c_{i},r})>\frac{ts}{2^{5}C_{3}}\qquad\mathrm{where}\qquad I=\left\{ i:H_{i}^{\frac{1}{2}}>\frac{ts}{2^{5}C_{3}} \right\}.
	\end{equation}
	For each $i\in I$ choose $x_{i}\in Q_{i}$ such that
	\begin{equation}
	g(x_{i})>\left(\frac{ts}{2^{5}C_{3}}\right)^{2}\cdot C_{1}N^d\cdot\mu(\mathsf{B}_{c_{i},r}).
	\end{equation}
	Then $\ref{HA_ref3}$ gives
	\begin{equation}
	\mu(\mathsf{B}_{x_{i},\frac{1}{N}})>\frac{g(x_{i})}{C_{1}N^d}>\frac{(ts)^{2}}{2^{10}C_{3}^{2}}\cdot\mu(\mathsf{B}_{c_{i},r}),
	\end{equation}
	and using \ref{HA_ref4},
	\begin{equation}
	\sum_{i\in I}\mu(\mathsf{B}_{x_{i},\frac{1}{N}})>\frac{(ts)^{3}}{2^{15}C_{3}^{3}}.
	\end{equation}
	The set $\tilde{X}={x_{i}:i\in I}$ visits each of the cubes $Q_{j}$ at most once. Thus it may be separated into $2^d$ subsets each of which never visits any neighboring $Q_{j}$ and is therefore $\frac{1}{M}$-separated. At least one of the $2^d$ subsets $X\subset\tilde{X}$ has
	\begin{equation}
	\mu(\bigcup_{x\in X}\mathsf{B}_{x,r})=\sum_{x\in X}\mu(\mathsf{B}_{x,\frac{1}{N}})>2^{-d}\sum_{i\in I}\mu(\mathsf{B}_{x_{i},\frac{1}{N}})>\frac{(ts)^{3}}{2^{d+15}\cdot C_{3}^{3}}
	\end{equation}
	This completes the proof of the proposition.
\end{proof}

\begin{prop}\label{MainProp}
	For $\lambda,\beta>0$ there exist $k\in\mathbb{N},C_1,C', L_{lb}>0$, such that if $L>L_{lb}$, $n \geq k$, and $S \subset [L,2L] $ is a $(\widetilde{C},\lambda)$-regular set for some $\widetilde C <L^{C_1}$ with $|S|>L^\beta$, and if 
	the measure $\mu_{n}=\nu_S^{*n}*\mu$ satisfies that for some $a\in\mathbb{Z}\backslash\{0\}$ and $t > L^{-C_1}$ that 
	\begin{equation}
	|\widehat{\mu}_{n}(a)|>t>0, 
	\end{equation}
	then there exists a $\frac{1}{M}$-separated set $X\subset\mathbb{T}$ with
	\begin{equation}
	\mu_{n-k}\left(\bigcup_{x\in X}\mathsf{B}_{x,\frac{1}{N}}\right)>C'\cdot t^{33\cdot 2^k},
	\end{equation}
	where $M=L^{k}|a|$ and $N=L^{k+\frac{1}{8^k}}|a|$.
\end{prop}
\begin{proof}
	Let $\epsilon_0$ be as in Lemma \ref{FinalBootstrap}. Set $\alpha_{ini}=0.99\beta$ and $\alpha_{high}=1-\epsilon_0$.
	Let $L_{lb}$ be the maximum of the value $L_{_1}$
	as in Lemma \ref{BootstrapLemma} and the value $L_1$ as in Lemma \ref{FinalBootstrap}.
	Let $C^*$ be such that the conditions of both lemmas, Lemma \ref{BootstrapLemma} and Lemma \ref{FinalBootstrap}, hold. We shall determine $C_1$ later in the proof.
	
	By Lemma \ref{InitialDimensionLemma} we have that for $\mu_n$
	\begin{equation}
	\mathcal{N}\left(\mathcal{F}(\mu_{n-1},\frac{t}{2})\cap[-N,N];M\right)\geq \frac{t}{2}\left(\frac{N}{M}\right)^{\beta}
	\end{equation}
	where $N=L|a|,M=|a|$. Since $t$ is bounded from below by $L^{-C_1}$ which will depend only on $\alpha_{ini}$ (and formally also on $\alpha_{high}$) then we can modify $L_{lb}$, if necessary, to be large enough such that the following holds,
	\begin{equation}
	\mathcal{N}\left(\mathcal{F}(\mu_{n-1},\frac{t}{2})\cap[-N,N];M\right)\geq \left(\frac{N}{M}\right)^{0.99\beta}.
	\end{equation}
	
	We now use our bootstrapping lemma, Lemma \ref{BootstrapLemma}, to obtain denser and denser sets of large Fourier coefficients. We finish by applying the Final Bootstrapping Lemma, \ref{FinalBootstrap}. The first step is actually checking if we can reach the conclusion of this lemma by applying once Lemma \ref{FinalBootstrap}.
	
	If $\alpha_{ini}>1-\epsilon_0$ then apply Lemma \ref{FinalBootstrap} and Proposition \ref{GranulationProp} to complete the proof. If $\alpha_{ini}\leq 1-\epsilon_0$ then we do the following. Let $\alpha_{inc}$ be as in Lemma \ref{BootstrapLemma} for the chosen values of $\alpha_{ini}, \alpha_{high}$. Let $k'=\left\lfloor (1-\epsilon_0-\alpha_{ini})/\alpha_{inc}\right\rfloor$ and $k=k'+1$. Let $C_1$ be such that if $L_{lb}^{-C^*}<t$ then $L_{lb}^{-C_1}<(t^{2^k}/4^{6^k})^4/128$. 
	
	Apply Lemma \ref{BootstrapLemma} $k'$ times to obtain
	\begin{equation} \label{MainPropInEq}
	\mathcal{N}\left(\mathcal{F}(\mu_{n-k'},t^{2^k}/4^{6^k})\cap[-N',N'];M'\right)> \left(\frac{N'}{M'}\right)^{1-\epsilon_0}.
	\end{equation}
	Apply Lemma \ref{FinalBootstrap} to obtain
	\begin{equation}
		\mathcal{N}\left(\mathcal{F}(\mu_{n-k'},(t^{2^k}/4^{6^k})^4/128)\cap[-N'',N''];M'\right)>c\cdot\left(\frac{t^{2^k}}{4^{6^k}}\right)^{10}\left(\frac{N'}{M'}\right),
	\end{equation}
	where $c$ is the constant in the conclusion of Lemma \ref{FinalBootstrap}.
	Apply Proposition \ref{GranulationProp} to complete the proof.
\end{proof}

\end{section}

\begin{section}{Proof of Theorem \ref{MainThm1}}
	\begin{proof}[Proof of Theorem \ref{MainThm1}]
		For $\lambda,\beta>0$, let $k,L_{lb},C_1,C'$ be as in the conclusion of Proposition~\ref{MainProp}; we will set $L_1$ later to be greater than $L_{lb}$. Let $\tau_0= \min\{C_1,\frac1{10\cdot 8^k}\}$. 
		
		For $\tau<\tau_0$,
		let $a\in\mathbb{Z}\backslash\{0\}$ be such that $|a|<L^\tau$ and such that $|\widehat{\nu_S^{*k}*\mu}(a)|>L^{-\tau}$.
		Apply Proposition \ref{MainProp} to the measure $\nu_S^{*k}*\mu$ to obtain a set $X_1$ which is $\frac{1}{M}$-separated and $\mu(\bigcup_{x\in X_1}\mathsf{B}_{x,\frac{1}{N}})>C'\cdot L^{-\tau\cdot 33\cdot 2^k}$ for some constant $C'$, where $M=L^{k+\frac{1}{10\cdot 8^k}}$, $N=L^{k+\frac{1}{8^k}}$ (the proposition gives the slightly stronger statement that $X_1$ is $\frac1{L^{k}|a|}$ separated and that $\mu(\bigcup_{x\in X_1}\mathsf{B}_{x,\frac{1}{N'}})>C'\cdot L^{-33 \cdot 2^k \tau}$ for $N'=L^{k+\frac{1}{8^k}}|a|$; this clearly implies what we use here). Let $\overline{X}_1=\bigcup_{x\in X_1}\mathsf{B}_{x,\frac{1}{N}}$. 
		Set 
		\[\mu^{(1)}_1=\mu\bigr|_{\mathbb{T}\backslash\overline{X}_1}\qquad\text{and}\qquad \mu^{(1)}_2=\mu\bigr|_{\overline{X}_1}.\]
		As long as there are large Fourier coefficients of the measure $\mu^{(1)}_1$ in the relevant range, we continue in a similar manner: for $a\in\mathbb{Z}\backslash\{0\}$ in the range $|a|<L^{\tau}$ such that $|\widehat{\nu_S^k*\mu^{(1)}_{1}}(a)|>L^{-\tau}$ obtain $X_2$ using Proposition \ref{MainProp}; in order to apply Proposition \ref{MainProp}, the measure $\mu^{(1)}_1$ is normalized so that the input is a probability measure ${\bar\mu}^{(1)}_1$, which only increases the Fourier coefficient, so
		\[
		|\widehat{\nu_S^k*{\bar\mu}^{(1)}_1}(a)|>|\widehat{\nu_S^k*\mu^{(1)}_{1}}(a)|>L^{-\tau}.
        \]
		We obtain a set $X_2$ which is $\frac{1}{M}$-separated and has the property that
		\[
		\mu^{(1)}_1(\bigcup_{x\in X_2}\mathsf{B}_{x,\frac{1}{N}})>C'\cdot L^{-33 \cdot 2^k \tau}\mu^{(1)}_1(\mathbb T).
		\]
		Let $\overline{X}_2=\bigcup_{x\in X_2}\mathsf{B}_{x,\frac{1}{N}}$.
		Set $\mu^{(2)}_1,\mu^{(2)}_2$ to be the following new measures:
		\begin{align*}
		\mu^{(2)}_1 &= \mu\bigr|_{\mathbb{T}\backslash (\overline{X}_1\cup\overline{X}_2)}\\
	     \mu_2^{(2)} &= \mu\bigr|_{\overline{X}_1\cup\overline{X}_2}.
		\end{align*}
		We repeat this step in an analogous manner, as long there is an $|a|<L^{\tau}$ for which $\bigl|\widehat{\mu_1^{(\ell)}}(a)\bigr| > L^{-\tau}$, obtaining a (finite, as we shall soon see) sequence of measures $\mu^{(\ell)}_1,\mu^{(\ell)}_2$ for $\ell=0,\dots,\ell_{max}$. Note that for every $\ell \leq \ell_{max}$	
		\[
		\mu^{(\ell)}_1(\mathbb T) \leq (1-C' \cdot L^{-33 \cdot 2^k \tau})^\ell \leq e^{-C' \cdot L^{-33 \cdot 2^k \tau} \ell}
		\]
		which in particular shows that $\ell_{max} < \infty$ as $\bigl|\widehat{\mu^{(\ell)}_1}(a)\bigr| \leq \mu^{(\ell)}_1(\mathbb T)$ for all $\ell$. Indeed, this shows that 
		\[
		e^{-C' \cdot L^{-33 \cdot 2^k \tau} \ell_{max}} \geq L^{-\tau}
		\]
		hence 
		$\ell_{max} < {C'}^{-1} (\log L) \cdot L^{33 \cdot 2^k \tau}\tau < L^{34 \cdot 2^k \tau}$ if $L_1$ is large enough.
	\end{proof}
	
\end{section}


\section*{Acknowledgments} 
The work presented in this paper comprises most of my PhD dissertation written at the Hebrew University. I wish to thank my advisor, Prof. Elon Lindenstrauss, for having me as a student, acquainting me with the techniques and ideas which consist of this work, guiding me and more. I am deeply grateful for that. 
	
I would also like to thank Prof. Barak Weiss and Prof. Tamar Ziegler, who served on my PhD committee, for following the process of the work and asking important questions. I would like to thank Prof. Mike Hochman for his remarks and help. In addition, I wish to thank the Hebrew University and the Einstein Institute for being such a pleasant home. Lastly, I wish to thank the anonymous referees who meticulously read my thesis. One of the referees summarized the work so well that I could not avoid using his descriptions in the introduction of this paper. 

\bibliographystyle{amsplain}


\begin{dajauthors}
\begin{authorinfo}[tomg]
Tom Gilat\\
Postdoctoral Fellow \\
Faculty of Engineering\\
Bar-Ilan University\\
Ramat Gan, Israel\\
tom.gilat\imageat{}biu\imagedot{}ac\imagedot{}il
\end{authorinfo}

\end{dajauthors}

\end{document}